\documentclass{amsart}
\usepackage{amsmath, amssymb, amsfonts, url, color, hyperref}

\newtheorem{theorem}{Theorem}[section]

\newtheorem{proposition}{Proposition}[section]
\newtheorem{corollary}{Corollary}[section]
\newtheorem{question}{Question}[section]
\newtheorem{definition}{Definition}[section]

\newtheorem{remark}{Remark}[section]

\numberwithin{equation}{section}

\newcommand{\CC}{\mathbb{C}}
\newcommand{\RR}{\mathbb{R}}

\newcommand{\cP}{\mathcal{P}}

\begin{document}

\title{Functions holomorphic along a $C^1$ pencil of 
holomorphic discs}

\author{Ye-Won Luke Cho and Kang-Tae Kim}



\thanks{Research of the authors is supported in part by NRF Grant
4.0019528 of The Republic of Korea (South).}

\begin{abstract}
The main purpose of this article is to present a generalization of Forelli's 
theorem for the functions holomorphic along a general pencil of 
holomorphic discs.  This generalizes the main result of \cite{JKS13} and the original Forelli's theorem, and furthermore, answers one of the problems posed in \cite{Chirka06}.
\end{abstract}

\maketitle

\section{Introduction and main theorems}

Denote by $B^n (p; r) := \{z \in \CC^n \colon \|z-p\| < r \}$ and by
$S^m := \{v \in \RR^{m+1} \colon \|v\|=1\}$.  With such notation,
the boundary of $B^n (0; 1)$ is $S^{2n-1}$.  Then the well known
Forelli's theorem \cite{Forelli77} states

\begin{theorem}[Forelli] \label{Forelli-original}
If a function $f\colon B^n (0; 1) \to \CC$ satisfies the following 
two conditions
\begin{enumerate}
\item $f\in C^{\infty}(0)$, meaning that for any positive integer $k$ there
exists an open neighborhood $V_k$ of the origin $0$ such that $f \in C^k 
(V_k)$, and
\item $f_v (\zeta) := f (\zeta v)$ is holomorphic on $B^1 (0;1)$ for every 
$v \in \CC^n$ with $\|v\|=1$,
\end{enumerate}
then $f$ is holomorphic on $B^n(0;1)$.
\end{theorem}

Deferring the historical counts to a later section, we state the 
main theorems of this article.   

\begin{definition}[Pencil]  
\textup{
For a nonempty open subset $U$ of $S^{2n-1}$, consider 
$P_0(U) := \{ \lambda u \colon \lambda \in B^1 (0;1), u \in U \}$.
By a \textit{standard pencil} of holomorphic discs at $0\in \mathbb{C}^n$ 
we mean the pair $(P_0(U), \psi)$ where $\psi (u, \lambda) := \lambda u$ 
for each $u \in U$ defines a holomorphic map 
$\psi(u, \cdot)\colon B^1(0;1) \to B^n (0;1)$.
}

\textup{
More generally, we say that a pair $(\cP (p, U, \Omega), \varphi)$ is a 
\textit{$C^k$ pencil of holomorphic discs at $p\in \Omega$} 
($k >0$ integer) if $\cP (p, U, \Omega) = \varphi (P_0 (U))$ and 
$\varphi\colon P_0 (U) \to \Omega$ is a $C^k$ diffeomorphism onto
its image that satisfies:
\begin{enumerate}
\item $\varphi (0) = p$, and 
\item $\lambda \to \varphi(\lambda u)\colon B^1 (0;1) \to \Omega$ 
defines a holomorphic map.
\end{enumerate}
}
\end{definition}
Note that for each $u\in U$, the map $\lambda\to \varphi(\lambda u)$ is a holomorphic embedding.

\begin{theorem} \label{CK20a}
If a complex valued function $f\colon B^n (0;1) \to \CC$ satisfies 
the following two properties
\begin{enumerate}
\item $f \in C^\infty (0)$, and
\item $\lambda \to f(\lambda u)$ defines a holomorphic function 
in $\lambda\in B^1 (0;1)$ for every $u \in U$ for 
some nonempty open subset $U$ of $S^{2n-1}$,
\end{enumerate}
then there exists $r>0$ such that $f$ is holomorphic on 
$B^n (0; r) \cup P_0 (U)$.
\end{theorem}

Notice that this, together with Hartogs' lemma (Sect.\ 5 of \cite{Chirka06}), 
implies Forelli's theorem above.

\begin{theorem} \label{CK20b}
If a domain $\Omega \subset \CC^n$ admits a $C^1$ pencil, say
$(\cP(p, U, \Omega), \varphi)$, of holomorphic discs at $p\in \Omega$,  then
any complex valued function $f\colon \Omega \to \CC$ 
satisfying the conditions
\begin{enumerate}
\item $f \in C^\infty (p)$, and 
\item $f$ is holomorphic along the pencil $(\cP(p, U, \Omega), \varphi)$, meaning
that $\lambda \to f(\varphi(\lambda u))$ is holomorphic 
in $\lambda \in B^1 (0;1)$ for any $u \in U$,
\end{enumerate}
is holomorphic on $B^n (0; r) \cup \cP(p, U, \Omega)$ for some
$r>0$.
\end{theorem}

Notice that this answers the second question in Section 6 
of \cite{Chirka06} (p.\ 219).  It also generalizes the following:

\begin{theorem}[Joo-Kim-Schmalz \cite{JKS13}]\label{JKS13}
If $U = S^{2n-1}$ and if $f\colon\Omega \to \CC$ satisfies the conditions
(1) and (2) of the hypotheses of the preceding theorem, then $f$ is holomorphic on $\Omega$.
\end{theorem}

\begin{remark} \textup{
The condition, which we denote by $f \in E(p)$, that $f$ admits a 
formal Taylor series at $p$ is weaker than $f \in C^\infty(p)$.  The 
precise definition that $f\in E(p)$ is as follows: for each positive integer 
$k$, there exists a polynomial $p_k (z,\bar z)$ with degree not more 
than $k$ satisfying
\[
f(z) - p_k (z,\bar z) = o(\|z\|^k),
\]
as $\|z\|$ tends to $0$.  We point out that the condition $f \in C^\infty (p)$ 
in the preceding theorems can be weakened to $f \in E(p)$ and $f$ is of 
class $C^1$ on a neighborhood of $p$. It will become evident, since 
those are the only assumptions we are working with throughout the paper.}
\end{remark}

\section{Structure of paper, and remarks} 

The original version of Forelli's theorem is concerned with the functions 
harmonic along the leaves of linear foliation of the ball $B^n (0;1)$ 
by the radial complex discs passing through the origin.  
Forelli proved that such a function,
say $\psi\colon B^n (0;1) \to \RR$ is pluriharmonic provided also
that $\psi \in C^\infty (0)$.  Then the proof arguments imply (with
only very minor adjustments) Theorem \ref{Forelli-original} 
as pointed out in \cite{Stoll80}.

There had been many attempts to weaken the condition 
$f \in C^\infty (0)$ for $f$ to finite differentiability. This may  have
sounded reasonable to a naive mind, as Hartogs' analyticity theorem
does not require any regularity assumptions beyond separate analyticity.
But no success was possible; numerous nonholomorphic examples of $f$ 
satisfying condition (2) of Theorem \ref{Forelli-original} were found, when 
condition (1) was replaced by the condition $f \in C^k$ for any finite 
$k$ \cite{KPS09}. On the other hand, there were quite a few successful 
attempts (more than 25 years after \cite{Forelli77}), starting with 
\cite{Chirka06}:  see \cite{KPS09, JKS13, JKS16}, just to name a few.  
Then there are related recent papers such as 
\cite{BoKu19, Krantz18}.  The second named author would like to 
acknowledge that professor A. Sadullaev \cite{Sadullaev} informed him of 
\cite{Madrak86}, even though we were unable to have an access to the paper.

We point out that the methods we introduce are elementary, to the
detail and different.  The starting point is with
the formal Taylor series of the given function $f:B^n(0;1)\to \mathbb{C}$ 
at the origin in both types of variables $z=(z_1, \ldots, z_n), 
\bar z = (\bar z_1, \ldots, \bar z_n)$. 
Then, under the assumptions of Theorem \ref{CK20a}, the formal 
power series turns out to be free of $\bar z$ terms.

Once this step is achieved, we use the answers to a question of Bochner
by Lelong \cite{Lelong51} (also by Zorn \cite{Zorn47} and Ree \cite{Ree49}).
A mild adjustment on Lelong's analysis \cite{Lelong51} leads us to an 
all dimensional principle which says that $f$ is holomorphic on 
$B^n (0;r)$ for some $r>0$.  Then the 
conclusion of Theorem \ref{CK20a} follows by Hartogs' lemma
(cf., e.g., \cite{Chirka06}).

At this juncture we, especially the second named author, would like to 
thank professors Takeo Ohsawa and Nessim Sibony for pointing out, in two
different occasions and with different suggestions, the possible relevance 
of \cite{Zorn47, Ree49, Lelong51}.
  
Once the standard pencil case is obtained, we show that the case of a general 
pencil of nonlinear Riemann surfaces follows by two important sources:
\begin{itemize}
\item[(1)] A variation (perhaps also minor) of the arguments of \cite{JKS13}. 
\item[(2)] Finding a standard pencil (in the given pencil) along which the given function is holomorphic.
\end{itemize}
On the other hand, we remark that we do not know how to carry out 
the arguments without the $C^1$ assumption on the foliation.

\section{Analysis with formal power series}

\subsection{Analysis with formal power series on standard pencil}

Note that the holomorphicity of $f:B^n(0;1)\to \CC$ along a standard 
pencil $(P_0(U),\psi)$ is the same as the following assumption:

\begin{itemize}
\item[(2')] $\bar E f \equiv 0$ on $P_0(U),$ where 
$E = \sum\limits_{k=1}^n z_k \frac\partial{\partial z_k}$.
\end{itemize}

Denote by $\CC[[z_1,\ldots,z_n,\bar{z}_1,\ldots,\bar{z}_n]]$ the ring 
of formal power series in the variables 
$z_1,\ldots,z_n,\bar{z}_1,\ldots,\bar{z}_n$ with coefficients in $\CC$.
Call $S \in \CC[[z_1,\ldots,z_n,\bar{z}_1,\ldots,\bar{z}_n]]$ \textit{of 
holomorphic type} if $S$ is free of variables $\bar{z}_1,\ldots,\bar{z}_n$.

\begin{proposition}\label{holo formal}
If $S = \sum C_{I}^{J}z^I\bar{z}^J$ is an element of   
$\CC[[z_1,\ldots,z_n,\bar{z}_1,\ldots,\bar{z}_n]]$ satisfying the equation 
$\bar{E}S\equiv 0$ on a 
nonempty open subset $V$ of $\mathbb{C}^n$, then $S$ is of 
holomorphic type, i.e., $C_{I}^{J}=0$ whenever $J\neq 0$.
\end{proposition} 

\begin{proof}
Recall the multi-index notation as follows: 
\[\alpha=(\alpha_1,\ldots,\alpha_n), ~z^{\alpha}=z_1^{\alpha_1}
\cdots z_n^{\alpha_n}~\text{and}~|\alpha|:=\alpha_1+\cdots+\alpha_n.\]
We first prove the proposition in the case that $S$ is a polynomial of finite 
degree. Note that $\bar{E}S\equiv 0$ on $\mathbb{C}^n$, since $\bar{E}S$ is 
real analytic. Consider the monomial term $S_{IJ}$ in $S$ of multi-degree 
$(I, J) := (i_1,\ldots,i_n,j_1,\ldots,j_n)$ where $J=(j_1,\ldots,j_n)\neq 0$. 
The monomial term in $\bar{E}S$ of multi-degree $(I, J)$ is precisely 
$\bar E (S_{IJ})$.   Then $\bar{E}S \equiv 0$ implies
\[
|J|C_{i_1,\ldots,i_n}^{j_1,\ldots,j_n}=0.
\] 
Consequently, $C_{i_1,\ldots,i_n}^{j_1,\ldots,j_n}=0$, whenever 
$J=(j_1,\ldots,j_n) \neq 0$.

Now let $S=\sum C_{I}^{J}z^I\bar{z}^J$ be any formal power series satisfying 
$\bar E S \equiv 0$ on $V$. Then, for each nonnegative integer $m$, 
we have 
\[
\bar{E}(S_m)=(\bar{E}S)_m=0~\text{on}~V,
\] 
where 
\[
S_m:=\sum_{|I|+|J|\leq m}C_{I}^{J}z^I\bar{z}^J.
\] 
Thus we conclude from the previous arguments on finite polynomials that 
$S$ is of holomorphic type.
\end{proof}

\begin{corollary}
Let $f\colon B^n(0;1) \to \mathbb{C}$ be a complex valued function 
satisfying the hypothesis of Theorem \ref{CK20a}. Then, the formal 
	Taylor series $S_f$ of $f$ at the origin is of holomorphic type.
\end{corollary}

\subsection{Bochner's problem on Radius of convergence}
Now the proof of Theorem \ref{CK20a} is reduced to establishing the 
existence of a region of convergence of $S_f$ at the origin. The goal of this 
subsection is to introduce Theorem \ref{Lelong} which is a criterion on the 
analyticity of formal power series of holomorphic type in two complex 
variables (all dimensional principle will be handled in the next section). The 
line of research concerned with the theorem originates from the following 
question of Bochner (cf.\ \cite{Zorn47}): 
  
\begin{question}
Let $S=\sum a_{ij}z_1^iz_2^j$ be a formal power series with complex 
coefficients such that every substitution of convergent power series with 
complex coefficients $z_1=\sum b_it^i$, $z_2=\sum c_it^i$ produces a 
convergent power series in $t$.  Is $S$ convergent on some neighborhood 
of $0\in \mathbb{C}^2?$
\end{question}

This question was answered affirmatively by Zorn in the following form.
   
\begin{theorem}[\cite{Zorn47}]\label{Zorn}
Let $S=\sum a_{ij}z_1^iz_2^j$ be a formal power series with complex 
coefficients such that for any $(a,b)\in \mathbb{C}^2$, $S(at,bt)$ is a power 
series in the complex variable $t$ with a positive radius of convergence. Then 
there exists a neighborhood $U$ of $0\in \mathbb{C}^2$ such that $S$ is 
holomorphic on $U$.
\end{theorem}

Zorn also remarked in \cite{Zorn47} that it would be interesting to know 
whether Bochner's conjecture holds in the `real case'. Ree \cite{Ree49} 
clarified the meaning of the real case and generalized the work of Zorn as 
follows:
     
\begin{theorem}[\cite{Ree49}]\label{Ree}
Let $S=\sum a_{ij}z^i_1z^j_2$ be a formal power series with complex 
coefficients such that for any $(a,b)\in \mathbb{R}^2$, $S(at,bt)$ is a power 
series in the complex variable $t$ with a positive radius of convergence. Then 
there exists a neighborhood $U$ of $0\in \mathbb{C}^2$ such that $S$ is 
holomorphic on $U$.
\end{theorem}
   
Lelong \cite{Lelong51} introduced the following definition to generalize the 
preceding theorems further:
      
\begin{definition}[\cite{Lelong51}]
\textup{$E\subset \mathbb{C}^2$ is called \textit{normal} if any 
formal power series $S\in \mathbb{C}[[z_1,z_2]]$ enjoying the property
that $S_{a,b}(t):=S(at,bt)\in \CC[[t]]$ has a positive radius of convergence 
for every $(a,b)\in E$ becomes holomorphic on some open neighborhood 
of the origin in $\mathbb{C}^2$.
}
\end{definition}

\subsection{Logarithmic capacity}
In the terminology of  \cite{Lelong51}, what Theorems \ref{Zorn} 
and \ref{Ree} say is that $\mathbb{R}$ and $\mathbb{C}$ are normal sets,
respectively.  For the sake of smooth exposition, we would like to 
cite the following definition.
   
\begin{definition}[cf., \cite{Ransford95}] \label{log capacity}
\textup{Let $\mu$ be a finite Borel measure on $\mathbb{C}$ with compact 
support. Then the $\textit{energy}~ I(\mu)$ of $\mu$ is given by 
$I(\mu)=\iint \log|z-w|d\mu(w) d{\mu}(z)$. The 
$\textit{logarithmic capacity}$ of a subset $E$ in $\mathbb{C}$ is 
defined to be $c(E):=\sup e^{I(\mu)}$, where the supremum is taken 
over the set of all probability measures on $\mathbb{C}$ with 
compact support in $E$.
}
\end{definition}

Then the theorem of Lelong in \cite{Lelong51} states:
   
\begin{theorem}[\cite{Lelong51}]\label{Lelong}
$E\subset \mathbb{C}^2$ is normal if, and only if, 
$E':=\{\frac{z_2}{z_1}\in \mathbb{C}\colon (z_1,z_2)\in E, z_1\neq 0\}$ 
is not contained in any $F_{\sigma}-$set $F$ with $c(F)=0$.
\end{theorem}

Notice that the proof of Theorem \ref{CK20a} in the case of complex 
dimension two follows by the arguments up to this point.
     
\section{Cases for standard pencils in all dimensions}

Now we generalize Theorem \ref{Lelong} to all dimensions and give a 
complete proof of Theorem \ref{CK20a}.  We shall follow the line 
of arguments by Lelong \cite{Lelong51} but start with the notion of 
capacity introduced by Siciak.
    
\begin{definition}[\cite{Siciak81}]
\textup{Denote by $\textup{PSH}(\mathbb{C}^n)$ the set of all 
plurisubharmonic functions on $\mathbb{C}^n$. For each positive 
integer $n$ and a subset $E$ of $\mathbb{C}^n$, define
\[ 
\mathcal{L}_n:=\{u\in \textup{PSH}(\mathbb{C}^n):\exists C_u\in \mathbb{R}
~ \text{such that}~u(z)\leq C_u+\textup{log}~(1+\|z\|)
~ \forall z\in \mathbb{C}^n \}, 
\]  
\[
V_E(z):=\textup{sup}\{u(z)\colon u\in \mathcal{L}_n ,u\leq 0 ~\text{on}~ E\},
~ \forall z\in \mathbb{C}^n.
\]
Then the $\textit{logarithmic capacity}$ of $E$ is defined to be 
\[
c_n(E):=e^{-\gamma_n(E)},
\]
where
}
\begin{align*}
\gamma_n(E)&:=\limsup\limits_{\|z\|\to \infty}(V^*_E(z)-\textup{log}\,\|z\|),\\
V^*_E(z)&:=\limsup\limits_{w\to z}V_E(w).
\end{align*}
\end{definition}
    
Notice that this concept of logarithmic capacity coincides  with 
the aforementioned logarithmic capacity in the case of $n=1$. 
The following theorem (cf. Sect.\ 3 of \cite{Siciak81}, p.172 of 
\cite{Klimek91}) plays an important role for our arguments:  
    
\begin{theorem}\label{sequence of pshfunctions}
Let $U$ be an open subset of $\mathbb{C}^n$ and $\{u_k\}$ a sequence of  
plurisubharmonic functions on $U$. Define functions $u,u^*$ on $U$ as
\[
u(z):=\limsup\limits_{k\to \infty}u_k(z),\quad 
u^*(z):=\limsup\limits_{U \ni w\to z}u(w).
\]
If $u$ is locally bounded from above, then the set 
$E:=\{z\in U\colon u(z)<u^*(z)\}$ has a vanishing logarithmic
capacity.
\end{theorem}

Now we investigate the behavior of the sequence 
$\{\frac{1}{k}\text{log}\,|P_k(z)|\}$ of plurisubharmonic functions.
\begin{proposition}\label{sequence of polynomials}
Let $\{P_k\}\subset \mathbb{C}[z_1,\ldots,z_n]$ be a sequence of 
polynomials with $\textup{deg}\,P_k\leq k$ for each positive integer $k$. 
Let $z=(z_1,\ldots,z_n)\in \mathbb{C}^n$, 
$r=(r_1,\ldots,r_n)\in {\mathbb{R}^n_{+}}$, where 
\[
{\mathbb{R}^n_{+}}:=\{(r_1,\ldots,r_n)\in \mathbb{R}^n\colon r_i>0
~\text{for each}~i = 1,\ldots,n\}.
\]
Also, define a sequence of plurisubharmonic functions 
\(
u_k(z):=\frac{1}{k} \textup{log}\, |P_k(z)|
\)
on $\mathbb{C}^n$ and their average 
\[
u^{r}_k(z):=\frac{1}{(2\pi)^n}\int_{0}^{2\pi}\cdots\int_{0}^{2\pi}
P_k(z_1+r_1e^{i\theta_1},\ldots,z_n+r_ne^{i\theta_n})\ d\theta_1
\cdots d\theta_n
\] 
on the distinguished boundary of the polydisc 
$P^n(z;r):=\{(w_1,\ldots,w_n)\in \mathbb{C}^n
\colon |w_i-z_i|<r_i~\text{for all}~ i =1,\ldots,n\}$ of polyradius $r$ 
centered at $z$. Then the following hold:\\
 	
\begin{enumerate}	
\item Let $r_0\in (0,\infty)$ be given. If $r,s\in \mathbb{R}^n_+$ 
and $r_i, s_i>r_0,~ \forall i = 1,\ldots,n$, then
\begin{equation}\label{estimate n}
|u^{r}_k(z)-u^{s}_k(w)|
< \frac{1}{r_0}(|r-s|+|z-w|)
\end{equation}
for all positive integer $k$ and $z,w\in \mathbb{C}^n$, where 
$|z-w|:=\sum_{i=1}^{n}|z_i-w_i|.$\\
\item Fix $r=(r_1,\ldots,r_n) \in \RR_+^n$ 
and let $\alpha_{r}=
\limsup_{k\to \infty}u^{r}_k(0)$. Then one and the only one of the following 
mutually exclusive cases is valid:
\\
\begin{enumerate}
\itemsep 0.4em
 		
\item $\alpha_{r}= -\infty$ and $\{u_k\}$ is uniformly bounded from above on each compact subset of $\mathbb{C}^n$.		
\item $\alpha_{r}= +\infty$ and $u(z)=\limsup_{k\to \infty}u_k(z)=\infty$ except on a set of vanishing logarithmic capacity.
 		
\item $-\infty<\alpha_{r}<\infty$ and $\{u_k\}$ is uniformly bounded from above on each compact subset of $\mathbb{C}^n$. \\
\end{enumerate}	
\end{enumerate}
\end{proposition}

\begin{proof} 
(1) In the case of $n=1$, the following inequality has been established 
in \cite{Lelong51}:
\begin{equation}\label{estimate 1}
|u^r_k(z)-u^s_k(w)|< \frac{1}{r_0}(|r-s|+|z-w|)
\end{equation}	 
for any positive integer $k$, $z,w\in \mathbb{C}$, $r,s>r_0$.

Let 
\begin{align*}
z&=(z_1,\ldots,z_n),~ w=(w_1,\ldots,w_n)\in \mathbb{C}^n,\\
r&=(r_1,\ldots,r_n),~ s=(s_1,\ldots,s_n)\in {\mathbb{R}^n_{+}}
\end{align*}
be given. Define
\[
z'_i=(w_1,\dots,w_i,z_{i+1},\dots,z_n),~
r'_i=(s_1,\dots,s_i,r_{i+1},\dots,r_n)
\]
for each $i\in \{0,\ldots,n\}$ and fix a positive integer $k$. Since $u_k$ 
is subharmonic in each variable separately, the following sequence of 
inequalities follows from the inequality (\ref{estimate 1}): 
\[
|u^{r'_{i-1}}_k(z'_{i-1})-u^{r'_i}_k(z'_i)|<\frac{1}{r_0}(|r_i-s_i|+|z_i-w_i|),
~ i \in\{1,\ldots,n\}.
\]
Then we obtain
\begin{align*}
|u^{r}_k(z)-u^{s}_k(w)|
&\leq 
\sum_{i=1}^{n} |u^{r'_{i-1}}_k(z'_{i-1})-u^{r'_i}_k(z'_i)| \\
&<
\sum_{i=1}^{n}\frac{1}{r_0}(|r_i-s_i|+|z_i-w_i|)\\
&=
\frac{1}{r_0}(|r-s|+|z-w|).
\end{align*}

(2-a) If $\alpha_{r}=-\infty$, then $u^{r}_k(0)\to -\infty$ as $k\to \infty$. 
Let $K$ be a compact subset of $\mathbb{C}^n$. By the inequality 
(\ref{estimate n}) and the sub-mean value inequality $u_k(z)\leq u^{r}_k(z)$,
\[
\lim\limits_{k\to \infty}u_k(z)= -\infty
\]
uniformly in $z\in K$.

(2-b) If $\alpha_{r}=\infty$, then there exists a subsequence 
$\{u^r_{n_k}(0)\}$ of $\{u^r_{k}(0)\}$ such that 
\[
\lim\limits_{k\to \infty}u^{r}_{n_k}(0)=\infty.
\]
Define a sequence $\{v_k(z)\}$ of plurisubharmonic functions on 
$\mathbb{C}^n$ by
\[v_k(z):=\frac{u_{n_k}(z)}{u^{r}_{n_k}(0)}.
\]
It follows from the inequality (\ref{estimate n}) that 
$\lim \limits_{k \to \infty}{v}_k^{r}(z)=1$ for all $z\in \mathbb{C}^n$. 
Set $v(z):=\limsup \limits_{k\to \infty}v_k(z)$ for any $z\in \mathbb{C}^n$. 
One can check directly that
\[
v^*(z):=\limsup \limits_{w\to z}v(w)
=\lim \limits_{r\to 0}(\limsup \limits_{k\to \infty}v_k^{r}(z))=1
\] 
for every $z \in \mathbb{C}^n$. Consequently,
\[
\{z\in \mathbb{C}^n\colon u(z)<\infty \} \subset \{z\in \mathbb{C}^n
\colon v(z)=0<1=v^*(z)\}.
\] 
Therefore, $u(z)= \infty $ for all $z\in \mathbb{C}^n$ except on a set of 
vanishing logarithmic capacity by Theorem \ref{sequence of pshfunctions}.

(2-c) Let $K$ be a compact subset of $\mathbb{C}^n.$ If $\alpha_{r}$ is 
finite, then it follows from the inequality (\ref{estimate n}) that $\{u^{r}_k(z)\}$ 
is uniformly bounded from above on $K$. Therefore, $\{u_k(z)\}$ is also 
uniformly bounded from above on $K$ since $u_k(z)\leq u^{r}_k(z)$ for all 
positive integer $k$ and $z \in \mathbb{C}^n$. 
\end{proof}

\begin{definition}	
\textup{A set $E\subset \mathbb{C}^n$ is called $\textit{normal}$
 if any formal 
power series $S\in \CC[[z_1,$ $\ldots,z_n]]$ for which 
$S_{a_1,\ldots,a_n}(t):=S(a_1t,\ldots,a_nt)\in \CC[[t]]$ has a positive 
radius of convergence $R_{(a_1,\ldots,a_n)}>0$ for every 
$ (a_1,\ldots,a_n)\in E$ becomes holomorphic on some open neighborhood 
of $0$ in $\mathbb{C}^n$.
}

\end{definition}

Now we use:

\begin{theorem}\label{generalized Lelong}
$E\subset \mathbb{C}^n$ is normal if $c_{n-1}(E')\neq 0$, where 
$E':=\{(\frac{z_2}{z_1},\ldots,\frac{z_n}{z_1})\in \mathbb{C}^{n-1}
\colon  (z_1,\ldots,z_n)\in E, z_1\neq 0 \}$.
\end{theorem}

In fact, more general result was known earlier. See \cite{LevenMol88} and 
the references therein.  But we only need the simple case stated here.  We
choose to provide here the following proof for the sake of smooth reading.

\begin{proof}
Let
\[
S=\sum a_{i_1,\ldots,i_n}{z_1}^{i_1}\cdots {z_n}^{i_n}\in \CC[[z_1,\ldots,z_n]]
\] 
be a formal power series for which 
$S_{a_1,\ldots,a_n}(t):=S(a_1t,\ldots,a_nt)$ has a positive radius of 
convergence $R_{(a_1,\ldots,a_n)}$ $>0$ for every $(a_1,\ldots,a_n)\in E$.
We are to show that $S$ is holomorphic on some open neighborhood 
of $0$ in $\mathbb{C}^n$.
Note that, for any $b=(b_1,\ldots,b_{n-1})\in E'$, the power series 
$S_{1,b_1,\ldots,b_{n-1}}$ converges absolutely and uniformly within
the radius $\frac12 R_{1,b_1,\ldots,b_{n-1}}$.  So it can be rearranged
as follows, with the same radius of convergence: 

\begin{align*}
S_{1,b_1,\ldots,b_{n-1}}(t)&=\sum_{k=0}^{\infty}\Big(
\sum_{i_1,\ldots,i_{n-1}=0}^{k}
a_{k-(i_1+\cdots+i_{n-1}),i_1,\ldots,i_{n-1}}
b_1^{i_1}\cdots b_{n-1}^{i_{n-1}}\Big) t^k \\
&=\sum_{k=0}^{\infty}P_k(b)t^k,
\end{align*}
where
\[
P_k(z):=\sum_{i_1,\ldots,i_{n-1}=0}^{k}a_{k-(i_1+\cdots+i_{n-1}),
i_1,\ldots,i_{n-1}}z_1^{i_1}\cdots z_{n-1}^{i_{n-1}}
\in \mathbb{C}[z_1,\ldots,z_{n-1}].
\]
Notice that $P_k$ is a polynomial of degree less than or equal to 
$k$. 
The root test implies
\[
\limsup\limits_{k\to \infty}\frac{1}{k}\text{log}\, |P_k(b)|=-\text{log}\,
R_{(1,b)}<\infty ~\text{for all}~ b\in E'.
\] 

Choose any ${r}=(r_0,\ldots,r_0) \in \RR_+^n$ and recall that 
$c_{n-1}(E')\neq 0$.  By Proposition \ref{sequence of polynomials}, 
there exists a constant $M>0$ such that 
\[
\frac{1}{k}\,\text{log}\,|P_k(b)|\leq \text{log}\,M,
\] 
or equivalently, 
\[
|P_k(b)|<M^k, ~\forall b \in P^{n-1}(0;2r),
\]
for any positive integer $k$.  Cauchy estimates implies 
\[
|a_{k-(i_1+\cdots+i_{n-1}),i_1,\ldots,i_{n-1}}|
\leq M^k r_0^{-{(i_1+\cdots+i_{n-1})}}
\]
for any multi-index $(i_1,\ldots,i_{n-1})$ satisfying  
$i_1+\cdots+i_{n-1}\leq k$. This yields that
\[
|a_{i_1,\ldots,i_{n}}z_1^{i_1}\cdots z_n^{i_n}|<(\tfrac{1}{2})^{k}
\] 
whenever
\[
|z_1|<\frac{1}{2M},|z_2|<\frac{r_0}{2M},~\cdots~,|z_n|<\frac{r_0}{2M} 
\textrm{ and } i_1+\cdots+i_n=k.
\]
Therefore, we have
\[
\sum_{i_1+\cdots+i_n=k}
|a_{i_1,\ldots,i_{n}}z_1^{i_1}\cdots z_n^{i_n}|<k^n 2^{-k}
\] 
for each positive integer $k$, $(z_1,\ldots,z_n)\in P^n(0;r')$, where 
$r'=(\frac{1}{2M},\frac{r_0}{2M},\ldots,\frac{r_0}{2M})\in \mathbb{R}^n_+$. 
Then we conclude from the Weierstrass $M$-test that the formal power series 
$S$ is convergent on $P^n(0;r')$.
\end{proof}

\textit{Proof of Theorem \ref{CK20a}}. Let 
$f\colon B^n(0;1)\to \mathbb{C} $ be a function that is smooth at the 
origin and holomorphic along a standard pencil $(P_0(U),\psi)$. Then the 
formal Taylor series $S_f$ is of holomorphic type by Corollary 
\ref{holo formal}. Note that given an open subset $U$ of $S^{2n-1}\subset 
\mathbb{C}^n$, the corresponding set $U'$ in 
$\mathbb{C}^{n-1}$ contains an open ball, say $B$, of a positive radius.
Note also that $0\neq c_{n-1}(B)\leq c_{n-1}(U')$. So, $S_f$ is 
holomorphic on $B^n(0;r)$ for some $r>0$ by 
Theorem \ref{generalized Lelong}. Now $f=S_f$ and moreover, 
Hartogs' lemma [\textit{Ibid}] implies that 
there exists a unique holomorphic extension of $f$ on $B^n(0;r)\cup P_0(U)$.
\hfill $\Box$

\section{Case of a general pencil---Proof of Theorem \ref{CK20b}}

\begin{definition} \textup{
By a \textit{subpencil} of a pencil $(\cP(p, U, \Omega), h)$ 
we mean a pencil of the form $(\cP(p, V, \Omega), h')$, 
where $h'$ is a restriction of $h$ to $P_0(V)\cap B^n(0;r)$ for some 
constant $r$ with $0<r\leq 1$ and an open subset $V$ of $U$. 
}
\end{definition}

The key turns out to be in finding an appropriate standard subpencil 
at $p$ along which $f$ is holomorphic. We therefore proceed in two steps
as follows:
\bigskip

\begin{narrower}

\textbf{Step 1.} There is a subpencil on whose underlying set 
$f$ is holomorphic. 

\textbf{Step 2.}  There is a standard subpencil of 
the pencil found in Step 1.

\end{narrower}
\bigskip

\noindent
Notice that it follows by Step 2 and Theorem \ref{CK20a} that 
there is an open neighborhood of $p$ on which $f$ is holomorphic.
This of course proves Theorem \ref{CK20b}, the desired final
conclusion.

\bigskip

\textbf{Step 1.} We may assume without loss of generality that $p=0$ in 
$\mathbb{C}^n$. Now we recapitulate the following result of \cite{JKS13}:

\begin{theorem}[\cite{JKS13}] \label{CR equation}
Let $(\cP(0, U, \Omega), h)$ be a $C^1$ pencil of holomorphic discs at the 
origin in $\mathbb{C}^n$. If $f\colon \Omega \to \mathbb{C}$ satisfies the 
following conditions
\begin{enumerate}
\item $f\in C^{\infty}(0)$, and
\item $f$ is holomorphic along the pencil,
\end{enumerate}
then for each $v\in U$, there exists a positive number $r_v>0$ 
such that $f$ satisfies the Cauchy-Riemann equation on 
$L_v:=\{h(zv)\in \Omega \colon z\in B^1(0;r_v) \}$.
\end{theorem}

\begin{proof}
We sketch the proof only for the case of $n=2$, even though it works in 
all dimensions, as in \cite{JKS13}. 

Let $v_0 \in U$ be given. One can choose coordinates around the origin 
in $\mathbb{C}^2$ such that $v_0$ becomes 
$(1,0)$. Fix a number $\epsilon_0>0$  
such that  
\[
\tilde{h}\colon B^1(0;\epsilon_0)\times B^1(0;\epsilon_0)\to \Omega,
\quad \tilde{h}(z_1,z_2):=h(z_1,z_1z_2)
\] 
is well-defined.
By a change of holomorphic coordinates, the map $\tilde{h}$ can be locally 
expressed as 
\[
\tilde{h}\colon B^1(0;\epsilon)\times B^1(0;\epsilon)\to \mathbb{C}, 
\quad \tilde{h}(z_1,z_2)=(z_1,k(z_1,z_2))
\] 
where $\epsilon$ is a positive real number and 
$k\colon B^1(0;\epsilon)\times B^1(0;\epsilon)\to \mathbb{C}$ 
satisfies the following conditions:

\begin{enumerate}
\itemsep 0.5mm 
\item $k(z_1,z_2)$ is $C^1$ in $z_1,z_2$ and holomorphic in $z_1$.
\item $k(0,z_2)=0$ for any $z_2\in B^1(0;\epsilon) $.
\item $\frac{\partial k(z_1,z_2)}{\partial z_1}\vert_{z_1=0}=z_2$ and 
hence $k(z_1,z_2)=z_1z_2+o(|z_1|).$
\end{enumerate}
Note that $\tilde{h}(z_1,0)=(z_1,0)$ represents $h_{v_0}\colon
z\in B^1(0;\epsilon)\to h(zv_0)\in \Omega.$

Let $f=f(z,w)$ be a function satisfying the given assumptions and 
let $F(z_1,z_2):=f(z_1,k(z_1,z_2))$. 
Then
\[
\frac{\partial F}{\partial z_2}=\frac{\partial f}{\partial w}
\frac{\partial k}{\partial z_2}+\frac{\partial f}{\partial \bar{w}}
\frac{\partial \bar{k}}{\partial z_2},
\]
\[
\frac{\partial F}{\partial \bar{z}_2}=\frac{\partial f}{\partial w}
\frac{\partial k}{\partial \bar{z}_2}+\frac{\partial f}{\partial \bar{w}}
\frac{\partial \bar{k}}{\partial \bar{z}_2}.
\]	
Since $\frac{\partial k}{\partial z_2}(z_1,0)$ and 
$\frac{\partial k}{\partial \bar{z}_2}(z_1,0)$ are holomorphic in 
$z_1$ (see p.1173--1174 of \cite{JKS13}), it follows that
\[
\Big(\frac{\partial k}{\partial \bar{z}_2}\frac{\partial F}{\partial z_2}
-\frac{\partial k}{\partial z_2}\frac{\partial F}{\partial \bar{z}_2}\Big)
\Big|_{(z_1,0)}
\] 
is also holomorphic in $z_1$.  Notice that this last is equal to 
\[
H(z_1)\frac{\partial f}{\partial \bar{w}}(z_1,0),
\]
where
\[
H(z_1):=\Big(\frac{\partial k}{\partial \bar{z}_2}
\frac{\partial \bar{k}}{\partial z_2}
-\frac{\partial k}{\partial z_2}\frac{\partial \bar{k}}{\partial \bar{z}_2}
\Big)(z_1,0).
\]
Let $G(z_1) = H(z_1)\frac{\partial f}{\partial \bar{w}}(z_1,0)$.
By a direct computation, $G$ turns out to coincide with $\bar{z}_1g(z_1)$, 
where $g$ is a smooth function. Hence the Taylor series of $G$ at 
$0\in \mathbb{C}$ vanishes identically, which implies that $G\equiv 0$.  
It holds also that $H(z_1)$ is nowhere zero in a punctured neighborhood
of $0$.  Altogether, it follows that there exists $r>0$ such that 
\[
\frac{\partial f}{\partial \bar{w}}(z_1,0)=0~\text{whenever}~ |z_1|<r,
\]
which implies
\[
0=\frac{\partial F}{\partial\bar{z}}(z_1,0)
= (\frac{\partial f}{\partial\bar{z}_1}+\frac{\partial f}{\partial\bar{w}}
\frac{\partial \bar{k}}{\partial\bar{z_1}})(z_1,0)
= \frac{\partial f}{\partial\bar{z}}(z_1,0) 
\]
for any $z_1$ with $|z_1|<r$. Therefore, $f$ satisfies the Cauchy-Riemann  
equation on $L_{v_0}$.
\end{proof}

Moreover, we have 

\begin{proposition}
Let $(\cP (0, U, \Omega),h)$ and $f:\Omega\subset \mathbb{C}^n \to \CC$ 
be given as in Theorem \ref{CR equation}. Then there exists a subpencil 
$(\cP (0, V, \Omega),h')$ which $f$ satisfies the Cauchy-Riemann 
equation at every point of the underlying set of.
\end{proposition}

\begin{proof}
For each positive integer ${\ell}$, define
\[
V_{\ell}:=\{v\in U\colon \frac{\partial f}{\partial \bar{z}_k}(h(zv))\equiv 0~ 
\text{for all}~ k\in \{1,\ldots,n\},~ z\in B^1(0;\frac{1}{\ell})\}.
\]
Since $f$ is $C^1$ on $B^n(0;r)$ for some $r>0$, each $V_{\ell}$ is closed 
and by Theorem \ref{CR equation}, 
\[
V=\bigcup_{\ell} V_{\ell}
\] 
where the union is taken over the set of all positive integers. By Baire 
category theorem, there exists a positive integer $m$ such that $V_m$ 
contains some nonempty open subset $V$ of $U$. 
Then $f$ satisfies the Cauchy-Riemann equation on
\[
\cP (0, V, \Omega):=\Big\{h(zv)\in \Omega\colon 
z\in B^1(0;\frac{1}{m}), v\in V\Big\},
\]
as desired.
\end{proof}

\textbf{Step 2.}
We are ready to complete the proof of Theorem \ref{CK20b}.  Recall that
it remains only to establish the existence of a standard subpencil of 
$\cP(0, V, \Omega)$ along which $f$ is holomorphic.

\begin{theorem}
Let $(\cP(0, V, \Omega),h)$ be a $C^1$ pencil of holomorphic discs at the 
origin of $\mathbb{C}^n$ and $W$ a relatively compact open subset of $V$. 
Then $P_0(W)\cap B^n(0;r)\subset\cP(0, V, \Omega)$ for some $r>0$. 
\end{theorem}

\begin{proof}
Suppose that 
\[
P_0(W)\cap B^n(0;\tfrac{1}{m})\not\subset \cP(0, V, \Omega)
\]
for any positive integer $m$. Fix $v \in W$. Since $h$ is a homeomorphism, 
there exists $s_v>0$ such that $zv\in h(B^1(0;1))$, for any $z\in \CC$
with $|z|<s_v$. 
Then there exist sequences
\[
\{z_k\}\subset B^1(0;1),~\{z'_k\}\subset B^1(0;1),
~\{v'_k\}\subset S^{2n-1}\backslash V,~ \{v_k\}\subset W
\]
satisfying the following conditions:
    
\begin{enumerate}
\item $z_kv_k=h(z'_kv'_k)\notin \cP(0, V, \Omega)$ for each 
positive integer $k$,
\item $\lim_{k\to\infty} z_k = 0 = \lim_{k\to\infty} z'_k$.
\end{enumerate}
Taking subsequences whenever necessary, we may assume without 
loss of generality that 
\[
\lim\limits_{k \to \infty}v'_k= v' \in S^{2n-1}  \text{ and }~ 
\lim\limits_{k \to \infty}v_k= v \in \overline{W}\subset V.
\]
Furthermore, we see that
\[
\big(\lim\limits_{k \to \infty}\tfrac{z_k}{z'_k}\big)v
=\lim\limits_{k \to \infty}\big\{\tfrac{h(z'_kv'_k)}{z'_k}\big\}=v'.
\]
Therefore, there exists a real number $\theta$ such that 
$e^{i\theta}v'=v\in V$.  Since $\lim_{k\to\infty} v'_k = v'$, there is 
a positive integer $N$ such that $e^{i\theta}v'_N\in V$.  This implies that
\[
z_Nv_N=h(z'_Nv'_N)=h(e^{-i\theta}z'_Ne^{i\theta}v'_N)\in \cP(0, V, \Omega),
\]
contradicting (1) above.  This completes the proof. 
\end{proof}

\section*{Data availability statement}

This article does not use any associated data.

\vspace{50pt}

Ye-Won Cho (\texttt{ww123hh@postech.ac.kr}) 

Kang-Tae Kim (\texttt{kimkt@postech.ac.kr})
\medskip

Department of Mathematics, 

Pohang University of Science and 
Technology (POSTECH), 

Pohang 37673, The Republic of Korea (South).

\end{document}